\documentclass[11pt,a4paper]{amsart}
\usepackage{amssymb,amsmath,amsfonts}

\headsep=1.2500cm \numberwithin{equation}{section}
\newcommand{\set}[1]{\left\{#1\right\}}

\newtheorem{Theorem}{Theorem}[section]
\newtheorem{Proposition}[Theorem]{Proposition}
\newtheorem{cor}[Theorem]{Corollary}
\newtheorem{lemma}[Theorem]{Lemma}
\theoremstyle{remark}

\newtheorem{Example}[Theorem]{Example}
\newtheorem{Remark}[Theorem]{Remark}

\begin{document}

\title{Approximate biprojectivity of certain semigroup algebras}

\author[A. Sahami]{A. Sahami}

\address{Faculty of Mathematics and Computer Science,
Amirkabir University of Technology, 424 Hafez Avenue, 15914
Tehran, Iran.}

\email{amir.sahami@aut.ac.ir}

\author[A. Pourabbas]{A. Pourabbas}
\email{arpabbas@aut.ac.ir}



\maketitle
%-------------------------------------------------------------

%%%%%%%%%%%%%%%%%%%%%%%%%%%%%%%%%%%%%%%%%%%%%%%%%%%%%%%%%%%%%%%%%%%%%%%%%
\begin{abstract}
In this paper, we investigate the notion of approximate biprojectivity for  semigroup algebras and for some Banach algebras related to semigroup algebras.
We show that  $\ell^{1}(S)$ is  approximately biprojective if and only if $\ell^{1}(S)$ is biprojective,
provided that $S$ is a uniformly locally finite inverse semigroup. Also for a Clifford semigroup $S$, we show that approximate biprojectivity $\ell^{1}(S)^{**}$ gives pseudo amenability of $\ell^{1}(S)$. We give a class of Banach algebras related to semigroup algebras which is not approximately biprojective.
\end{abstract}
\section{Introduction}
Amenable Banach algebras were  introduced by Johnson in \cite{Joh}. In fact a Banach algebra $A$ is amenable, if every continuous  linear derivation $D:A\rightarrow X^{*}$ is inner, for every Banach $A$-bimodule $X$. He showed that $A$ is amenable Banach algebra if and only if $A$ has an approximate diagonal, that is a bounded net $(m_{\alpha})_{\alpha}$ in $A\otimes_{p}A$ such that $\pi_{A}(m_{\alpha})a\rightarrow a$ and $a\cdot m_{\alpha}-m_{\alpha} \cdot a\rightarrow 0,$ for every $a\in A.$

Most important notions related to amenability in the theory of homological  Banach algebras  are biflatness and by biprojectivity which introduced by Helemskii in \cite{hel}. Indeed, $A$ is called biflat (biprojective), if
there exists a bounded $A$-bimodule morphism $\rho:A\rightarrow
(A\otimes_{p}A)^{**}$ ($\rho:A\rightarrow A\otimes_{p}A$) such that
$\pi^{**}\circ\rho$ is the canonical embedding of $A$ into $A^{**}$
($\rho$ is a right inverse for $\pi_{A}$), respectively.

Recently some modificated notions of amenability like approximate amenability and pseudo amenability introduced, see \cite{ghah app}, \cite{ghah app1} and \cite{ghah pse}. In order to these new notions, approximate homological notions like approximate biprojective Banach algebras and approximate biflat Banach algebras introduced, for more information see \cite{zhang} and \cite{sam}.

Kanuith {\it et al.} in  \cite{kan} introduced   the  notion of left $\phi$-amenable Banach
algebras, where $\phi$ is a character on that Banach algebra. Later on the concepts of  left $\phi$-contractible and character amenable Banach algebras were defined, see \cite{sang} and \cite{nas}.

Semigroup algebras are very important Banach algebras. The amenability of these Banach algebras studied in many papers, common reference about the amenability of semigroup algebras is  \cite{dales semi}. Recently modificated notions like pseudo-amenability, pseudo-contractibility and approximate amenability of semigroup algebras   have been investigated, see \cite{rost}, \cite{rost1} and \cite{rost 2}. Indeed they studied  pseudo-amenability, pseudo-contractibility and approximate amenability of $\ell^{1}(S)$, where $S$ is an inverse group, band semigroup and etc. Biflatness and biprojectivity of semigroup algebras were another problem which investigated in  \cite{choi} and \cite{rams}. In fact in  \cite{rams} author showed that for an inverse semigroup $S$, $\ell^{1}(S)$ is biflat (biprojective) if and only if each maximal subgroup $S$ is amenable (finite) and $S$ is uniformly locally finite semigroup, respectively.  The question is what will happen if semigroup algebra $\ell^{1}(S)$ is approximate biprojective?

In this paper we use left $\phi$-contractibility and left $\phi$-amenability to investigate approximate biprojectivity of semigroup algebras. We show that approximate biprojectivity of $\ell^{1}(S)$ implies the finiteness of $S$, for some classes of semigroups.
We study  approximate biprojectivity
of the second dual of semigroup algebras. We show that for Clifford semigroup $S$, approximate biprojectivity of $\ell^{1}(S)^{**}$ implies that $\ell^{1}(S)$ is pseudo-amenable. We give a criteria which shows that some triangular Banach algebras related to semigroup algebras are not approximate biprojective.
%------------------------------------------------------------------------------------------------------------------------------------------
%%%%%%%%%%%%%%%%%%%%%%%%%%%%%%%%%%%%%%%%%%%%%%%%%%%%%%%%%%%%%%%%%%%%%%%%%%%%%%%%%%%%%%%%%%%%%%%%%%%%%%%%%%%%%%%%%%%%%%%%%%%%%%%%%%%%%%%%%%%
%------------------------------------------------------------------------------------------------------------------------------------------
\section{Preliminaries}
Let $A$ be a Banach algebra. We recall that if $X$ is a Banach
$A$-bimodule, then  $X^{*}$ is also a Banach
$A$-bimodule via  the following actions
$$(a\cdot f)(x)=f(x\cdot a) ,\hspace{.25cm}(f\cdot a)(x)=f(a\cdot x ) \hspace{.5cm}(a\in A,x\in X,f\in X^{*}). $$

Throughout, the
character space of $A$ is denoted by $\Delta(A)$,  that is, all
non-zero multiplicative linear functionals on $A$. Let $\phi\in
\Delta(A)$. Then $\phi$ has a unique extension   $\tilde{\phi}\in\Delta(A^{**})$
 which is defined by $\tilde{\phi}(F)=F(\phi)$ for every
$F\in A^{**}$.

Let $A$ and  $B$ be  Banach algebras. The projective tensor product
of $A$ with $B$ is denoted by $A\otimes_{p}B$. The Banach algebra
$A\otimes_{p}A$ is a Banach $A$-bimodule via the following actions
$$a\cdot(b\otimes c)=ab\otimes c,~~~(b\otimes c)\cdot a=b\otimes
ca\hspace{.5cm}(a, b, c\in A).$$

We recall that $\Delta(A\otimes_{p}B)=\{\phi\otimes\psi|
\,\phi\in\Delta(A), \psi\in\Delta(B)\}$, where
$\phi\otimes\psi(a\otimes b)=\phi(a)\psi(b), $ for every $a\in A$
and $b\in B.$ We denote $\pi_{A}:A\otimes_{p}A\rightarrow A$ for the
product morphism which specified by $\pi_{A}(a\otimes b)=ab.$

Let $\{A_{\alpha}\}_{\alpha\in \Gamma}$ be a collection of Banach algebras.
Then we define the  $\ell^{1}$-direct sum of $A_{\alpha}$ by  $$\ell^{1}-\oplus_{\alpha \in \Gamma} A_{\alpha}=\{(a_{\alpha})\in \prod_{\alpha\in \Gamma}A_{\alpha}:\sum_{\alpha\in\Gamma}||a_{\alpha}||<\infty\}.$$ It is easy to verify that
$$\Delta(\ell^{1}-\oplus_{\alpha\in\Gamma} A_{\alpha})=\{\oplus\phi_{\beta}:\phi_{\beta}\in\Delta(A_{\beta}),\beta\in\Gamma\},$$
where $\oplus\phi_{\beta}((a_{\alpha})_{\alpha\in \Gamma})=\phi_{\beta}(a_{\beta})$ for every
$(a_{\alpha})_{\alpha\in \Gamma}\in \ell^{1}-\oplus_{\alpha \in \Gamma} A_{\alpha}$ and every $\beta\in\Gamma$.

Let $A$ be a Banach algebra and  let $\Lambda$ be a non-empty set.
The set of all
$\Lambda\times\Lambda$ matrixes $(a_{i,j})_{i,j}$ which entries come
from $A$ is denoted by $\mathbb{M}_{\Lambda}(A)$. With matrix multiplication and the following norm
$$||(a_{i,j})_{i,j}||=\sum_{i,j}||a_{i,j}||<\infty,$$
$\mathbb{M}_{\Lambda}(A)$ is a Banach algebra.
$\mathbb{M}_{\Lambda}(A)$ belongs to the class of $\ell^{1}$-Munn
algebras.  The map
$\theta:\mathbb{M}_{\Lambda}(A)\rightarrow A\otimes_{p}
\mathbb{M}_{\Lambda}(\mathbb{C})$ defined by
$\theta((a_{i,j}))=\sum_{i,j}a_{i,j}\otimes E_{i,j}$ is an isometric
algebra isomorphism, where $(E_{i,j})$ denotes the matrix unit of
$\mathbb{M}_{\Lambda}(\mathbb{C})$.
 Also it is well-known that $
\mathbb{M}_{\Lambda}(\mathbb{C})$ is a  biprojective Banach algebra \cite[Proposition 2.7]{rams}.

The main  reference for the semigroup theory is \cite{how}. Let $S$
be a semigroup and let $E(S)$ be the set of its idempotents. A
partial order on $E(S)$ is defined by
$$s\leq t\Longleftrightarrow s=st=ts\quad (s,t\in E(S)).$$
If $S$ is an inverse semigroup, then there exists a partial
order on $S$ which is coincide with the partial order on $E(S)$.
Indeed
$$s\leq t\Longleftrightarrow s=ss^{*}t\quad (s,t\in
S).$$ For every  $x\in S$, we denote $(x]=\{y\in
S|\,y\leq x\}$. $S$ is called locally finite (uniformly
locally finite) if for each $x\in S$, $|(x]|<\infty\,\,(\sup\{|(x]|\,:\,x\in
S\}<\infty)$, respectively.

Suppose that $S$ is an inverse semigroup. Then  the maximal subgroup
of $S$ at $p\in E(S)$ is denoted by
$G_{p}=\{s\in S|ss^{*}=s^{*}s=p\}$. For an inverse semigroup $S$ there exists a relation
$\mathfrak{D}$ such that $s\mathfrak{D}t$ if and only if there
exists $x\in S$ such that $ss^{*}=xx^{*}$ and $t^{*}t=x^{*}x$. We denote $\{\mathfrak{D}_{\lambda}:\lambda\in \Lambda\}$ for the collection of $\mathfrak{D}$-classes and $E(\mathfrak{D}_{\lambda})=E(S)\cap \mathfrak{D}_{\lambda}.$
An inverse semigroup $S$ is called Clifford if for each
$s\in S$, there exists $s^{*}$ such that $ss^{*}=s^{*}s.$
%------------------------------------------------------------------------------------------------------------------------------------------
%%%%%%%%%%%%%%%%%%%%%%%%%%%%%%%%%%%%%%%%%%%%%%%%%%%%%%%%%%%%%%%%%%%%%%%%%%%%%%%%%%%%%%%%%%%%%%%%%%%%%%%%%%%%%%%%%%%%%%%%%%%%%%%%%%%%%%%%%%%
%------------------------------------------------------------------------------------------------------------------------------------------
\section{Approximate biprojectivity of semigroup algebras}
We recall that a Banach algebra $A$ is approximately biprojective, if there exists a net $(\rho_{\alpha})_{\alpha}$ of continuous $A$-bimodule morphism from $A$ into $A\otimes_{p}A$ such that $\pi_{A}\circ \rho_{\alpha}(a)\rightarrow a$ for every $a\in A.$ For more details see \cite{zhang}.

A Banach algebra $A$ is called
left $\phi$-amenable (left $\phi$-contractible), where $\phi\in\Delta(A)$,
if there exists $m\in A^{**}(m\in A)$ such that $am=\phi(a)m$ and
$\tilde{\phi}(m)=1 \,\,(\phi(m)=1),$ for every $a\in A$,
respectively, see \cite{kan} and \cite{nas}.

 A Banach algebra $A$ is called
pseudo-contractible if there exists a not necessarily bounded net
$(m_{\alpha})_{\alpha}$ in $A\otimes_{p}A$ such that $a\cdot
m_{\alpha}=m_{\alpha}\cdot a$ and
$\lim_{\alpha}\pi_{A}(m_{\alpha})a=a$, for every $a\in A.$ For
further details see \cite{ghah pse}.

We remind that $S$ is a left amenable (a right amenable) semigroup
if there exists an element $m\in \ell^{1}(S)^{**}$ such that
$$s\cdot m=m\,\,( m=m\cdot s),\hspace{.25cm} ||m||=m(\phi)=1\hspace{1cm} (s\in S),$$
where $\phi$ is the augmentation character of  $\ell^{1}(S)$,
respectively. The semigroup $S$ is called amenable,  if it is both
left and right amenable.

\begin{Proposition}\label{result}
Let $S$ be a semigroup and let $Z(S)$ be a non-empty set. Then
\begin{enumerate}
\item [(i)] If $\ell^{1}(S)^{**}$ is approximately biprojective, then $S$ is amenable;
\item [(ii)] If  $\ell^{1}(S)$ is approximately biprojective and $S$ has left or right unit, then $S$ is finite.
\end{enumerate}
\end{Proposition}
\begin{proof}
 (i) Let $\ell^{1}(S)^{**}$ be approximately biprojective. Then
there exists a net of bounded  $\ell^{1}(S)^{**}$-bimodule morphisms
$\rho_{\alpha}:\ell^{1}(S)^{**}\rightarrow
\ell^{1}(S)^{**}\otimes_{p}\ell^{1}(S)^{**}$ such that
$\pi_{\ell^{1}(S)^{**}}\circ\rho_{\alpha}(a)\rightarrow a$ for every
$a\in \ell^{1}(S)^{**}$. By \cite[Lemma 1.7]{gha loy} there exists a
bounded linear map $\psi:\ell^{1}(S)^{**}\otimes_{p}
\ell^{1}(S)^{**}\rightarrow (\ell^{1}(S)\otimes_{p}
\ell^{1}(S))^{**}$ such that for $a,b\in \ell^{1}(S)$ and $m\in
\ell^{1}(S)^{**}\otimes_{p} \ell^{1}(S)^{**}$, the following holds;
\begin{enumerate}
\item [($\ast$)] $\psi(a\otimes b)=a\otimes b $,
\item [($\ast$$\ast$)] $\psi(m)\cdot a=\psi(m\cdot a)$,\qquad
$a\cdot\psi(m)=\psi(a\cdot m),$
\item [($\ast$$\ast$$\ast$)] $\pi_{\ell^{1}(S)}^{**}(\psi(m))=\pi_{\ell^{1}(S)^{**}}(m).$
\end{enumerate} It is easy to see that
$\rho^{\alpha}_{0}=\psi\circ\rho_{\alpha}|_{\ell^{1}(S)}:\ell^{1}(S)\rightarrow
(\ell^{1}(S)\otimes_{p} \ell^{1}(S))^{**}$ is a net of bounded
$\ell^{1}(S)$-bimodule morphisms.

Let $\phi$ be the augmentation character on $\ell^{1}(S)$ and
$\tilde{\phi}$ be its extension to $\ell^{1}(S)^{**}$, for every
$s_{0}\in Z(S)$ we have  $\phi(\delta_{s_{0}})=1.$ Since
$\tilde{\phi}\circ\pi^{**}_{\ell^{1}(S)}\circ\rho^{\alpha}_{0}(a)\rightarrow
\phi(a)$ for every $a\in \ell^{1}(S)$, by taking
$m_{\alpha}=\rho^{\alpha}_{0}(\delta_{s_{0}})$, one can easily see
that $a\cdot m_{\alpha}=m_{\alpha}\cdot a$ and
$\tilde{\phi}\circ\pi^{**}_{\ell^{1}(S)}(m_{\alpha})\rightarrow 1.$
We can assume $\tilde{\phi}\circ\pi^{**}_{\ell^{1}(S)}(m_{\alpha})=
1$ by considering
$\frac{m_{\alpha}}{\tilde{\phi}\circ\pi^{**}_{\ell^{1}(S)}(m_{\alpha})}$
instead of  $m_{\alpha}$.
So we have $a\cdot m_{\alpha}=m_{\alpha}\cdot a=\phi(a)m_{\alpha}$ and $\tilde{\phi}\circ\pi^{**}_{\ell^{1}(S)}(m_{\alpha})=1$.

Set $\mu_{\alpha}=\pi^{**}_{\ell^{1}(S)}(m_{\alpha})$, so $\mu_{\alpha}\in \ell^{1}(S)^{**}$,  $\delta_{s}\mu_{\alpha}=\mu_{\alpha}\delta_{s}=\mu_{\alpha}$
and $\tilde{\phi}(\mu_{\alpha})=\mu_{\alpha}(\phi)=1,$  hence $S$ is an amenable semigroup, see the proof of \cite[Corollary 2.10]{rost1}.

 (ii) Suppose that $(\rho_{\alpha})_{\alpha}$ is a net of
continuous $\ell^{1}(S)$-bimodule morphism such that
$\lim_{\alpha}\pi_{\ell^{1}(S)}\circ\rho_{\alpha}(a)=a$, for every
$a\in A.$ Set $M_{\alpha}=\rho_{\alpha}(\delta_{s_{0}})$, where
$s_{0}\in Z(S)$, then it is easy to see that $a\cdot
M_{\alpha}=M_{\alpha}\cdot a$ and
$\phi(\pi_{\ell^{1}(S)}(M_{\alpha}))\rightarrow 1,$ where $\phi$ is
the augmentation character.
Without loss of generality we may assume that  $a\cdot
M_{\alpha}=M_{\alpha}\cdot a=\phi(a)M_{\alpha}$ and
$\phi(\pi_{\ell^{1}(S)}(M_{\alpha}))= 1$. So $\ell^{1}(S)$ is left
and right $\phi$-contractible. Now using the same arguments as in
the  \cite[Corollary 2.10]{rost1}, we can find $m\in \ell^{1}(S)$
such that $$\delta_{s}m=m\delta_{s}=m.$$ If $e_{l}$ is a left
identity for $S$, then for every $s\in S$, we have
$$m(s)=m(e_{l}s)=\delta_{s}m(e_l)=m(e_l),$$ that is,
$m\in\ell^{1}(S)$ is a constant function on $S$, so $S$ must be
finite.
\end{proof}
\begin{Remark}
Note that the converse of the previous Proposition (i) is not always
true. To see this let $S=G$ be an infinite, discrete and amenable
group. Suppose  that $\ell^{1}(G)^{**}$ is approximately
biprojective. Since $\ell^{1}(G)$ is unital, then $\ell^{1}(G)^{**}$
is unital too. Hence by \cite[Proposition 3.8]{ghah pse}
$\ell^{1}(G)^{**}$ is a pseudo-contractible Banach algebra. So
$\ell^{1}(G)^{**}$ is  pseudo-amenable. Applying \cite[Proposition
4.2]{ghah pse} $G$ must be finite which is a contradiction.
\end{Remark}
\begin{Example}
Let $S=\set{{\left(\begin{array}{cc} 0&a\\
0&b\\
\end{array}
\right)| a,b\in\mathbb{C}}}$. With the matrix multiplication $S$ is a semigroup. We claim that $\ell^{1}(S)$ is not approximately biprojective.
 We go toward a contradiction and suppose that $\ell^{1}(S)$ is approximately biprojective.
 Let $s_{0}={\left(\begin{array}{cc} 0&0\\
0&0\\
\end{array}
\right)}\in Z(S)$ and let $s_{1}={\left(\begin{array}{cc} 0&1\\
0&1\\
\end{array}
\right)}$ be a right unit for $S.$ Now the hypothesis of the  previous Proposition(ii) holds. So $S$ is finite which is a contradiction.

Similarly for $S=\set{{\left(\begin{array}{cc} a&b\\
0&c\\
\end{array}
\right)| a,b, c\in\mathbb{C}}}$ or $S=\set{{\left(\begin{array}{cc} a&b\\
c&d\\
\end{array}
\right)| a,b, c,d\in\mathbb{C}}}$, $\ell^{1}(S)$ is not approximately biprojective.

Moreover if $S=\set{{\left(\begin{array}{cc} 0&a\\
0&0\\
\end{array}
\right)| a\in\mathbb{C}}}$, then  $S$ is amenable, but $\ell^{1}(S)$
is not approximately biprojective. To see this we suppose  that
$\ell^{1}(S)$ is approximately biprojective. Then there exists a net
of $\ell^{1}(S)$-bimodule morphism
$(\rho_{\alpha})_{\alpha}:\ell^{1}(S)\rightarrow
\ell^{1}(S)\otimes_{p}\ell^{1}(S)$
such that $\pi_{\ell^{1}(S)}\circ\rho_{\alpha}(a)\rightarrow a,$ for every $a\in \ell^{1}(S).$ Set $s_{0}={\left(\begin{array}{cc} 0&0\\
0&0\\
\end{array}
\right)}$ and $s_{1}={\left(\begin{array}{cc} 0&1\\
0&0\\
\end{array}
\right)}$. Let $\rho_{\alpha}(\delta_{s_{1}})=\sum_{i=1}^{\infty}a^{\alpha}_{i}\otimes b^{\alpha}_{i}$, for some nets $\{a^{\alpha}_{i}\}$ and $\{b^{\alpha}_{i}\}$ in $\ell^{1}(S)$. Since $\delta_{s}\delta_{s^{\prime}}=\delta_{s_{0}}$ for every $s, s^{\prime}$ in $S$, there exists a net  $\{x_{\alpha}\}$ in $\mathbb{C}$ such that $$\pi_{\ell^{1}(S)}\circ\rho_{\alpha}(\delta_{s_{1}})=x_{\alpha}\delta_{s_{0}},$$ then $$\pi_{\ell^{1}(S)}\circ\rho_{\alpha}(\delta_{s_{1}})=x_{\alpha}\delta_{s_{0}}\nrightarrow \delta_{s_{1}},$$
which is a
contradiction.
\end{Example}
\begin{Proposition}
Let $S$ be a semigroup. If $\ell^{1}(S)^{**}$ is pseudo-contractible, then $S$ is amenable.
\end{Proposition}
\begin{proof}
Let $\ell^{1}(S)^{**}$ be pseudo-contractible. By \cite[Theorem 1.1]{alagh1}, $\ell^{1}(S)^{**}$ is left and right $\phi$-contractible,
 for every $\phi\in\Delta(\ell^{1}(S))$. By \cite[Proposition 3.5]{nas} $\ell^{1}(S)$ is left and right $\phi$-amenable,
 for every $\phi\in\Delta(\ell^{1}(S))$ including  the augmentation character, so with  similar argument as in  the proof of Proposition \ref{result},
 one can show that $S$ is an amenable semigroup.
\end{proof}
The following lemma is similar to \cite[Proposition 2.2]{rams} which we omit the proof.
\begin{lemma}\label{lem}
Let $A$ and $B$ be  Banach algebras. Suppose that $A$ is unital and
$B$ has a non-zero idempotent.
If $A\otimes_{p}B$ is approximately biprojective,  then $A$ is
approximately biprojective.
\end{lemma}
\begin{Theorem}
Let $S$ be an inverse semigroup which $(E(S),\leq)$ is
uniformly locally finite. Then $\ell^{1}(S)$ is approximately biprojective if and only if $\ell^{1}(S)$ is biprojective.
\end{Theorem}
\begin{proof}
Let $\ell^{1}(S)$ be  approximate biprojective.  Then there exists a net $(\rho_{\alpha})_{\alpha}$ of continuous
$\ell^{1}(S)$-bimodule morphism from $\ell^{1}(S)$ into $\ell^{1}(S)\otimes_{p}\ell^{1}(S)$ such that $\pi_{\ell^{1}(S)}\circ \rho_{\alpha}(a)\rightarrow a$,
for every $a\in \ell^{1}(S)$. Since $S$ is a uniformly locally finite semigroup, by \cite[Theorem 2.18]{rams} we have   $$\ell^{1}(S)\cong
\ell^{1}-\bigoplus\{\mathbb{M}_{E(\mathfrak{D}_{\lambda})}(\ell^{1}(G_{p_{\lambda}}))\},$$
where $\mathfrak{D}_{\lambda}$ is a $\mathfrak{D}$-class and $G_{p_{\lambda}}$ is a maximal subgroup of $S$ at $p_{\lambda}$.

Let $P_{p_{\lambda}}:\ell^{1}(S)\rightarrow
\mathbb{M}_{E(\mathfrak{D}_{\lambda})}(\ell^{1}(G_{p_{\lambda}}))$
be a homomorphism which is dense range. It is easy to see that
$P_{p_{\lambda}}$ is a bounded $\ell^{1}(S)$-bimodule morphism.
Define $$\eta_{\alpha}=P_{p_{\lambda}}\otimes
P_{p_{\lambda}}\circ\rho_{\alpha}|_{\mathbb{M}_{E(\mathfrak{D}_{\lambda})}(\ell^{1}(G_{p_{\lambda}}))}:
\mathbb{M}_{E(\mathfrak{D}_{\lambda})}(\ell^{1}(G_{p_{\lambda}}))\rightarrow
\mathbb{M}_{E(\mathfrak{D}_{\lambda})}(\ell^{1}(G_{p_{\lambda}}))\otimes_{p}\mathbb{M}_{E(\mathfrak{D}_{\lambda})}(\ell^{1}(G_{p_{\lambda}})).$$
It is easy to see that $(\eta_{\alpha})_{\alpha}$ is a net of
$\mathbb{M}_{E(\mathfrak{D}_{\lambda})}(\ell^{1}(G_{p_{\lambda}}))$-bimodule
morphism which satisfied
\begin{equation}\label{equ}
\begin{split}
\pi_{\mathbb{M}_{E(\mathfrak{D}_{\lambda})}(\ell^{1}(G_{p_{\lambda}}))}\circ\eta_{\alpha}(a)&=\pi_{\mathbb{M}_{E(\mathfrak{D}_{\lambda})}(\ell^{1}(G_{p_{\lambda}}))}
\circ P_{p_{\lambda}}\otimes P_{p_{\lambda}}\circ\rho_{\alpha}|_{\mathbb{M}_{E(\mathfrak{D}_{\lambda})}(\ell^{1}(G_{p_{\lambda}}))}(a)\\
&=P_{p_{\lambda}}\circ\pi_{\ell^{1}(S)}
\circ\rho_{\alpha}|_{\mathbb{M}_{E(\mathfrak{D}_{\lambda})}(\ell^{1}(G_{p_{\lambda}}))}(a)\rightarrow a,
\end{split}
\end{equation}
for every $a\in \mathbb{M}_{E(\mathfrak{D}_{\lambda})}(\ell^{1}(G_{p_{\lambda}})).$
Therefore $ \mathbb{M}_{E(\mathfrak{D}_{\lambda})}(\ell^{1}(G_{p_{\lambda}}))$ is an approximately biprojective Banach algebra.
By Lemma \ref{lem} it is easy to see that $\ell^{1}(G_{p_{\lambda}})$ is approximately biprojective. Then by \cite[Proposition 3.8]{ghah pse}   $\ell^{1}(G_{p_{\lambda}})$ is pseudo-contractible, hence $G_{p_{\lambda}}$ is finite. Then $\ell^{1}(S)$ is biprojective by the main result of \cite{rams}.

Converse is clear.
\end{proof}
\begin{Theorem}
Let $S=\cup_{e\in E(S)}G_{e}$ be a Clifford semigroup such that
$E(S)$ is uniformly locally finite. If $\ell^{1}(S)^{**}$ is approximately biprojective, then $\ell^{1}(S)$ is pseudo-amenable.
\end{Theorem}
\begin{proof}
Let  $\ell^{1}(S)^{**}$ be  approximately biprojective.  Then there exists a net $(\rho_{\alpha})_{\alpha}$ of continuous $\ell^{1}(S)^{**}$-bimodule
morphism from $\ell^{1}(S)^{**}$ into $\ell^{1}(S)^{**}\otimes_{p}\ell^{1}(S)^{**}$ such that $\pi_{\ell^{1}(S)^{**}}\circ \rho_{\alpha}(a)\rightarrow a$
 for every $a\in \ell^{1}(S)^{**}$.  Since $S$ is a uniformly locally finite semigroup, by \cite[Theorem 2.16]{rams}
  $\ell^{1}(S)\cong \ell^{1}-\oplus_{e\in E(S)}\ell^{1}(G_{e}).$ Let $x_{e}$ be a unit element of $\ell^{1}(G_{e})$ and
  let $\phi\in\Delta(\ell^{1}(G_{e}))$.  It is well-known
that the  maps $b\mapsto x_{e}b$ and $b\mapsto bx_{e}$ are
$w^{*}-w^{*}$-continuous on $\ell^{1}(S)^{**}$. Then for every
$a\in\ell^{1}(S)^{**}$, we have $ax_{e}=x_{e}a$
 and $\tilde{\phi}(x_{e})=1,$ where $\tilde{\phi}\in\Delta(\ell^{1}(S)^{**})$ is the extension of $\phi$. Define $m^{e}_{\alpha}=\rho_{\alpha}(x_{e})\in\ell^{1}(S)^{**}\otimes \ell^{1}(S)^{**}$. Using \cite[Lemma 1.7]{gha loy}
   we can consider
$m^{e}_{\alpha}$ in $(\ell^{1}(S)\otimes_{p}\ell^{1}(S))^{**}$. It is easy to see that $a\cdot m_{\alpha}^{e}=m_{\alpha}^{e}\cdot a$ and
 $\tilde{\phi}\circ\pi^{**}_{\ell^{1}(S)}(m_{\alpha}^{e})= 1$, for every $a\in \ell^{1}(S))^{**}$. Applying \cite[Proposition 2.2]{sah1}
  $\ell^{1}(S)^{**}$ is left $\tilde{\phi}$-amenable. By \cite[Proposition 3.4]{kan} $\ell^{1}(S)$ is left $\phi$-amenable.
   Since $\phi|_{\ell^{1}(G_{e})}$ is non-zero, by \cite[Lemma 3.1]{kan} $\ell^{1}(G_{e})$ is left $\phi$-amenable. Then by similar argument as
   in the proof of \cite[Theorem 2.1.8]{run}, $G_{e}$ is amenable, for every $e\in E(S)$. To finish the proof apply \cite[Theorem 3.7]{rost}.
\end{proof}
\begin{Example}\label{rem}
\begin{enumerate}
\item [(i)] There exists a pseudo-amenable Banach algebra which is not approximately biprojective. To see this, let $G$ be an infinite amenable group.
 Then by \cite[Proposition 4.1]{ghah pse} $\ell^{1}(G)$ is pseudo-amenable.
  Suppose that $\ell^{1}(G)$ is  approximately biprojective. Since $\ell^{1}(G)$ is unital, using the same argument as in the proof of previous Proposition,
 one can show that $\ell^{1}(G)$ is left $\phi$-contractible. Then by \cite[Theorem 6.1]{nas} $G$ is finite which is a contradiction. Hence $\ell^{1}(G)$ is not approximately biprojective.
\item [(ii)] There exists an approximately biprojective Banach algebra which
is not pseudo-contractible. To see this, let
$A=\mathbb{M}_{\Lambda}(\mathbb{C})$, where $\Lambda$ is an infinite
set. By \cite[Proposition 2.7]{rams} $A$ is biprojective, so $A$ is
approximately biprojective. On the other hand if $A$ is
pseudo-contractible, then $A$ has a central approximate identity.
Therefore by  \cite[Theorem 2.2]{rost1} $\Lambda$ is finite, which
is a contradiction.
 \item [(iii)]Now we give a semigroup algebra which is
approximately biprojective but it is not pseudo-contractible. Let
$S$ be a right zero semigroup, that is, $st=t$ for every $s,t \in
S$, and let $|S|\geq 2$. Let $\phi$ be the augmentation character on
$\ell^{1}(S)$, so for every $f,g\in \ell^{1}(S)$ we have $f*
g=\phi(f)g$. One can see that $\ell^{1}(S)$ is biprojective, hence
it is approximately biprojective, but if $\ell^{1}(S)$ is
pseudo-contractible, then $\ell^{1}(S)$ has a right approximate
identity $(e_{\alpha})$. Consider $f_{0}\in\ell^{1}(S)$  such  that
$\phi(f_{0})=1$, so
\begin{equation}\label{eq1}
f_{0}=\lim_{\alpha} f_{0}\ast e_{\alpha}=\lim_{\alpha}\phi(f_{0})e_{\alpha}=\lim_{\alpha}e_{\alpha},
\end{equation}
that is $f_{0}$ is a right unit for $\ell^{1}(S)$.
On the other hand
\begin{equation}\label{eq2}
g\ast f_{0}=\lim_{\alpha} g\ast e_{\alpha}=\phi(g)f_{0},
\end{equation}
for every $g\in \ell^{1}(S).$ Let $s$ be an arbitrary element of
$S$. Then by (\ref{eq1}) and (\ref{eq2}), we have
$\delta_{s}=\delta_{s}\ast f_{0}=f_{0}$ which implies that $|S|=1$.
Therefore a contradiction reveals.
\end{enumerate}

Note that  example (iii) shows that  the hypothesis $Z(S)\neq\emptyset$ in the Proposition \ref{result}(ii) is necessary. Because if we
 consider a right zero semigroup $S$ with $|S|=\infty$, then $Z(S)=\emptyset$ and $S$ has a left identity. One can show that $\ell^{1}(S)$ is approximately biprojective but $S$ is not finite.

 Zhang in \cite{zhang} gives  an example of approximately biprojective Banach algebra which is
not biprojective.
\end{Example}

\begin{Theorem}\label{app give phi}
Let $A$ be an approximately biprojective Banach algebra with a left
approximate identity (right approximate identity) and let
$\phi\in\Delta(A)$. Then $A$ is left $\phi$-contractible(right
$\phi$-contractible), respectively.
\end{Theorem}
\begin{proof}
Suppose that $A$ is  approximately biprojective. Then there exists a net of $A$-bimodule morphisms $(\rho_{\alpha})_{\alpha}$ from $A$ into
$A\otimes_{p}A$ such that $\pi_{A}\circ\rho_{\alpha}(a)\rightarrow a,$ for every $a\in A$. Let $L=\ker\phi$.
Define $\eta_{\alpha}:=id_{A}\otimes q\circ\rho_{\alpha}:A\rightarrow A\otimes_{p}\frac{A}{L}$, where $q$ is a quotient map.
It is easy to see that $\eta_{\alpha}$ is a left $A$-module morphism, for every $\alpha$. Since $A$ has a left approximate identity, $\overline{AL}=L$,
so for every $l\in L$, there exist $a\in A$ and $l^{'}\in L$
such that $l=al^{'}$. Also since for every $l\in L$, $q(l)=0$ and $(\rho_{\alpha})$ is a net of $A$-bimodule morphism,
 we have $$\eta_{\alpha}(l)=id_{A}\otimes
q\circ\rho_{\alpha}(l)=id_{A}\otimes
q\circ\rho_{\alpha}(al^{'})=id_{A}\otimes
q\circ(\rho_{\alpha}(a)\cdot l^{'})=0.$$
Thus
$\eta_{\alpha}$ can be dropped on $\frac{A}{L}$, for every $\alpha$.
So we can see that $\eta_{\alpha}:\frac{A}{L}\rightarrow
A\otimes_{p}\frac{A}{L}$ is  a left $A$-module morphism.

We define a character $\overline{\phi}$ on $\frac{A}{L}$  by
$\overline{\phi}(a+L)=\phi(a)$, for every $a\in A.$ Consider
$\gamma_{\alpha}=id_{A}\otimes \overline{\phi}\circ
\eta_{\alpha}:\frac{A}{L}\rightarrow A$. Since
\begin{equation}
\begin{split}
\gamma_{\alpha}(a\cdot x+L)=id_{A}\otimes\overline{\phi}\circ\eta_{\alpha}(ax+L)&=id_{A}\otimes\overline{\phi}\circ\eta_{\alpha}(ax)\\
&=a\cdot ( id_{A}\otimes\overline{\phi}\circ\eta_{\alpha})(x),
\end{split}
\end{equation}
and $\eta_{\alpha}$ is a left $A$-module morphism, $\gamma_{\alpha}$ is a left $A$-module morphism.  Note that $(\gamma_{\alpha})$ is a net of non-zero maps. To see this consider
\begin{equation}
\begin{split}
\phi\circ\gamma_{\alpha}(x+L)=\phi\otimes\overline{\phi}\circ\eta_{\alpha}(x+L)&=\phi\otimes\overline{\phi}\circ\eta_{\alpha}(x)\\
&=\phi\circ\pi_{A}\circ\rho_{\alpha}(x)\rightarrow \phi(x)\neq 0,
\end{split}
\end{equation}
for every $x\in A.$ Pick $x_{0}$ in $A$ such that $\phi(x_{0})=1.$ Define $m_{\alpha}=\gamma_{\alpha}(x_{0}+L)$. Then we have $$\phi(m_{\alpha})=\phi\circ \gamma_{\alpha}(x_{0}+L)=\phi\circ\pi_{A} \circ\rho_{\alpha}(x_{0})\rightarrow \phi(x_{0})=1.$$ Consider
\begin{equation}
\begin{split}
ax_{0}+L=(a-\phi(a)x_{0}+\phi(a)x_{0})x_{0}+L&=ax_{0}-\phi(a)x_{0}^{2}+\phi(a)x_{0}^{2}+L\\
&=\phi(a)x_{0}^{2}+L\\
&=\phi(a)x_{0}+L,
\end{split}
\end{equation}
since $x_{0}^{2}-x_{0}\in L.$ Therefore
$$am_{\alpha}=a\gamma_{\alpha}(x_{0}+L)=\gamma_{\alpha}(ax_{0}+L)=\phi(a)\gamma_{\alpha}(x_{0}+L)=\phi(a)m_{\alpha}.$$
Replacing $(m_{\alpha})$ with
$(\frac{m_{\alpha}}{\phi(m_{\alpha})})$, we can assume that
$am_{\alpha}=\phi(a)m_{\alpha}$ and $\phi(m_{\alpha})=1$. Then $A$
is left $\phi$-contractible, see \cite[Theorem 2.1]{nas}.
\end{proof}
\begin{Remark}
Existence of a left approximate identity  is essential for previous Theorem, which we cannot omit  it. To see this let $A=\set{{\left(\begin{array}{cc} 0&a\\
0&b\\
\end{array}
\right)| a,b\in\mathbb{C}}}$. With matrix operation $A$ becomes a Banach algebra. It is easy to see that $A$ is a biprojective Banach algebra.
 Then $A$ is approximately biprojective. If $A$ has a left approximate identity, then an easy calculation shows that $\dim A=1$ which is impossible.
 On the other hand if  we define
$\phi(\left(\begin{array}{cc} 0&a\\
0&b\\
\end{array}
\right))=b$. It is easy to see that $\phi$ is a character on $A$. One can show that  $A$ is left $\phi$-contractible if and only if $\dim A=1$ which is impossible.
\end{Remark}
At the following result we extend \cite[Corollary 2.10]{rost1}(ii), to the approximate biprojective case.
\begin{cor}
Let $S$ be a semigroup with a left unit. If  $\ell^{1}(S)$ is approximately biprojective with a right approximate identity, then $S$ is finite.
\end{cor}
\begin{proof}
Let $\ell^{1}(S)$ be approximately biprojective. By Theorem \ref{app give phi}, $\ell^{1}(S)$ is left and right $\phi$-contractible for every $\phi\in\Delta(\ell^{1}(S))$. Follow the same arguments as in the proof of \cite[Corollary 2.10]{rost1} to finish the proof.
\end{proof}
\begin{Remark}
Let $S$ be a bicyclic semigroup, that is, $S$ is a semigroup, generated by two elements $p$ and $q$ which $pq=e$ for a  unit element $e.$
 Then $\ell^{1}(S)$ is a unital Banach algebra. Using the previous Corollary, one can see that $\ell^{1}(S)$ is not approximately biprojective.

Consider the semigroup $S=\mathbb{N}_{\vee}$,  with semigroup
operation $m\vee n=\max\{m,n\}$, where $m$ and $n$ are in
$S$. It is easy to see that $\ell^{1}(S)$ is a unital Banach algebra with unit $\delta_{1}$. Since $S$ is an infinite semigroup, by previous Corollary $\ell^{1}(S)$ is not approximately biprojective.
\end{Remark}
Suppose that $A$ and $B$ are Banach algebras and $M$ is a Banach
$(A,B)$-module. The matrix
algebra $T=\left(\begin{array}{cc} A&M\\
0&B\\
\end{array}
\right)$ is called a triangular Banach algebra which equipped with
the norm
$\|\left(\begin{array}{cc} a&m\\
0&b\\
\end{array}
\right)\|_{T}=\|a\|_{A}+\|m\|_{M}+\|b\|_{B}$ for $a\in A$, $m\in M$
and $b\in B$.

In  \cite[Corollary 3.3]{sat} the authors  showed that some triangular Banach algebras are  not biprojective at all. Here at the following theorem we are going to extend this result to the approximately biprojective case.
\begin{Theorem}\label{triangular}
Let $A$ be a Banach algebra with a left approximate identity and let $\phi\in\Delta(A)$. Then  $T=\left(\begin{array}{cc} A&A\\
0&A\\
\end{array}
\right)$ is not approximately biprojective.
\end{Theorem}
\begin{proof}
We are going toward a contradiction and suppose that $T$ is approximately biprojective. Define a character $\psi_{\phi}$ on $T$ by
 $\psi_{\phi}(\left(\begin{array}{cc} a&x\\
0&b\\
\end{array}
\right))=\phi(b)$, for every  $a,x$ and $b$ in $A$. Since $A$ has a left approximate identity, by Theorem \ref{app give phi}, $T$ is a left $\psi_{\phi}-$contractible Banach algebra. Set $I=\left(\begin{array}{cc} 0&A\\
0&A\\
\end{array}
\right)$. Clearly $I$ is a closed ideal in $T$ which $\psi_{\phi}|_{I}\neq 0$. Then by  \cite[Proposition 3.8]{nas} $I$ is left $\psi_{\phi}$-contractible too. Then there exists $\left(\begin{array}{cc} 0&i\\
0&j\\
\end{array}
\right)\in I$ such that
\begin{equation}\label{eq4}
\left(\begin{array}{cc} 0&a\\
0&b\\
\end{array}
\right)\left(\begin{array}{cc} 0&i\\
0&j\\
\end{array}
\right)=\psi_{\phi}(\left(\begin{array}{cc} 0&a\\
0&b\\
\end{array}
\right))\left(\begin{array}{cc} 0&i\\
0&j\\
\end{array}
\right)=\phi(b)\left(\begin{array}{cc} 0&i\\
0&j\\
\end{array}
\right)
\end{equation}
and $$\psi_{\phi}\left(\begin{array}{cc} 0&i\\
0&j\\
\end{array}
\right)=\phi(j)=1,
$$
for every $a, b\in A.$ Suppose that $(e_{\alpha})_{\alpha}$ is the left approximate identity of $A$. Let $a\in \{e_{\alpha}\}_{\alpha}$ and $b$ be an arbitary element of $\ker\phi$. Put $a$ and $b$ in (\ref{eq4}) we have $aj=\phi(b)i=0$. This implies that $e_{\alpha}j=0$ for every $\alpha$. Since $e_{\alpha}$ is an approximate identity for $A $, we have $j=0.$ On the other hand $\phi(j)=1$ which is a contradiction.
\end{proof}
Consider the semigroup $\mathbb{N}_{\wedge}$,  with the semigroup
operation $m\wedge n=\min\{m,n\}$, where $m$ and $n$ are in
$\mathbb{N}$. Let $w:\mathbb{N}_{\wedge}\rightarrow \mathbb{R}^{+}$ be a weight, that is a function which $w(st)\leq w(s)w(t)$, for every $s,t \in S$, for the further details see \cite{dales}.  We recall that for every weight $\ell^{1}(\mathbb{N_{\wedge}})$ has an approximate identity, see \cite[Proposition 3.3.1]{dales}. Also $\Delta(\ell^{1}(\mathbb{N_{\wedge}}),w)$ consists
precisely of the all functions
$\phi_{n}:\ell^{1}(\mathbb{N_{\wedge}},w)\rightarrow \mathbb{C}$
defined by
$\phi_{n}(\sum_{i=1}^{\infty}\alpha_{i}\delta_{i})=\sum_{i=n}^{\infty}\alpha_{i}$
for every $n\in\mathbb{N}$.
\begin{cor}
Let $S=\mathbb{N_{\wedge}}$ and $w$ be a weight on $S$. Then $T=\left(\begin{array}{cc} \ell^{1}(S,w)&\ell^{1}(S,w)\\
0&\ell^{1}(S,w)\\
\end{array}
\right)$ is not approximately biprojective.
\end{cor}
\begin{proof}
It is well-known that for every weight $\ell^{1}(S,w)$ has an approximate identity. Then $T$ has an approximate identity. Now apply previous Theorem, to finish the proof.
\end{proof}
\begin{cor}
Let $S=\mathbb{N_{\wedge}}$. Then $T=\left(\begin{array}{cc} \ell^{1}(S)^{**}&\ell^{1}(S)^{**}\\
0&\ell^{1}(S)^{**}\\
\end{array}
\right)$ is not approximately biprojective.
\end{cor}
\begin{proof}
We go toward a contradiction  and suppose that $T$ is approximately biprojective.
Since $\ell^{1}(S)$ has a bounded approximate identity, see \cite[Proposition 3.3.1]{dales}, $\ell^{1}(S)^{**}$ has a right unit. So $T$ has a right unit. Let $\phi\in \Delta(\ell^{1}(S)^{**})$ and  $\psi_{\phi}(\left(\begin{array}{cc} a&x\\
0&b\\
\end{array}
\right))=\phi(a)$, for every  $a,x$ and $b$ in $\ell^{1}(S)^{**}$.   Apply Theorem \ref{app give phi}, $T$ is right $\psi_{\phi}$-contractible. Set $I=\left(\begin{array}{cc} \ell^{1}(S)^{**}&\ell^{1}(S)^{**}\\
0&0\\
\end{array}
\right)$. Then $I$ is right $\psi_{\phi}|_{I}$-contractible. But follow the similar arguments as in the proof of Theorem \ref{triangular} implies that  $I$ is not right $\psi_{\phi}|_{I}$-contractible which is a contradiction.
\end{proof}
\begin{cor}
Let $S$ be a right zero semigroup. Then $T=\left(\begin{array}{cc} \ell^{1}(S)&\ell^{1}(S)\\
0&\ell^{1}(S)\\
\end{array}
\right)$ is not approximately biprojective.
\end{cor}
\begin{proof}
Since $S$ is a right zero semigroup, $\ell^{1}(S)$ has a left unit. Then $T$ has a left unit too. By Theorem \ref{triangular}, $T$ is not approximately biprojective.
\end{proof}
\begin{Proposition}
Let $S=\cup_{e\in E(S)}G_{e}$ be a Clifford semigroup such that
$E(S)$ is uniformly locally finite. Then $T=\left(\begin{array}{cc} \ell^{1}(S)^{**}&\ell^{1}(S)^{**}\\
0&\ell^{1}(S)^{**}\\
\end{array}
\right)$ is not approximately biprojective.
\end{Proposition}
\begin{proof}
It is well-known that  $\ell^{1}(S)\cong \ell^{1}-\oplus_{e\in E(S)}\ell^{1}(G_{e}).$ Let $x_{e}$ denote for unit element of $\ell^{1}(G_{e})$. It is easy to see that $\delta_{x_{e}}$ commutes with every elements of $\ell^{1}(S)$. Since two maps $b\mapsto \delta_{x_{e}}b$ and $b\mapsto b\delta_{x_{e}}$ are
$w^{*}-w^{*}$-continuous on $\ell^{1}(S)^{**}$, where $b\in \ell^{1}(S)^{**}$,  $\delta_{x_{e}}$ also commutes with every elements of $\ell^{1}(S)^{**}$. Consider the element $t=\left(\begin{array}{cc} \delta_{x_{e}}&0\\
0&\delta_{x_{e}}\\
\end{array}
\right)\in T$, it is east to see that $t$ commutes with every element of $T$. Let $\phi$ be the augmentation character on $\ell^{1}(S)$ and  $\tilde{\phi}$ its extension to $\ell^{1}(S)^{**}$ and $\psi_{\tilde{\phi}}$ be the character on $T$ which defined in the proof of Theorem \ref{triangular} with respect to $\tilde{\phi}.$ Now go toward a contradiction and suppose that $T$ is approximate biprojective. Follow the same arguments as in the proof of Proposition \ref{result}, $T$ is left and right $\psi_{\tilde{\phi}}$-contractible. Let $I=\left(\begin{array}{cc} 0&\ell^{1}(S)^{**}\\
0&\ell^{1}(S)^{**}\\
\end{array}
\right)$. It is easy to see that $I$ is a closed ideal of $T$. Then by \cite[Proposition 3.8]{nas} $I$ is left and right $\psi_{\tilde{\phi}}$-contractible. Hence there exists $t_{1}$ and $t_{2}$ in $I$ such that $at_{1}=\psi_{\tilde{\phi}}(a)t_{1}$, $t_{2}a=\psi_{\tilde{\phi}}(a)t_{2}$ and $\psi_{\tilde{\phi}}(t_{1})=\psi_{\tilde{\phi}}(t_{2})=1,$ for every $a\in I.$ Define $m=t_{1}t_{2}\in I$, then  there exists element $i$ and $j$ in $\ell^{1}(S)^{**}$ such that $m=\left(\begin{array}{cc} 0&i\\
0&j\\
\end{array}
\right)$. It is easy to see that
\begin{equation}\label{eq6}
am=ma,\quad \psi_{\tilde{\phi}}(m)=1,
\end{equation}
for every $a\in I.$ Set $a=\left(\begin{array}{cc} 0&x\\
0&y\\
\end{array}
\right)$, where $x,y\in \ell^{1}(S)^{**}$  and put $a$ in (\ref{eq6}). Then we have
\begin{equation}\label{eq7}
xj=iy,\quad \tilde{\phi}(j)=1,
\end{equation}
for every  $x,y\in \ell^{1}(S)^{**}$. Set $x=\delta_{x_{e}}$ and $y$ be any element of $\ker\tilde{\phi}$. Put these $x$ and $y$ in (\ref{eq7}), and take $\tilde{\phi}$ on equation  $xj=iy$ it implies that $\tilde{\phi}(j)=0$ which is a contradiction.
\end{proof}
\begin{Proposition}
Let $G$ be a locally compact group. Then $\ell^{1}(S)\otimes_{p} M(G)$ is approximately biprojective if and only if $G$ is finite,
 where $S$ is the semigroup which is defined  in the Example \ref{rem}.
\end{Proposition}
\begin{proof}
Suppose that $\ell^{1}(S)\otimes_{p} M(G)$  is approximately biprojective. Since $M(G)$ is unital and $\ell^{1}(S)$ has a left identity.
 Then by Theorem \ref{app give phi}  $\ell^{1}(S)\otimes_{p} M(G)$ is left $\phi$-contractible for every $\phi\in\Delta(\ell^{1}(S)\otimes_{p} M(G))$.
 Apply \cite[Theorem 3.14]{nas} to show that $M(G)$ is character contractible, hence by \cite[Corollary 6.2]{nas} $G$ is finite.

Converse holds By  biprojectivity of $\ell^{1}(S)$ and
\cite[Proposition 2.4]{rams}.
\end{proof}

\begin{Remark}
We want to give some Banach algebras which is never approximately biprojective. Consider the semigroup $\mathbb{N}_{\vee}$,  with semigroup
operation $m\vee n=\max\{m,n\}$, where $m$ and $n$ are in
$\mathbb{N}$. Authors in \cite[Example 3.5]{sah2} showed that for every   weight $w:\mathbb{N_{\vee}}\rightarrow \mathbb{R}^{+}$,
 $\ell^{1}(\mathbb{N_{\vee}}, w)$ is not pseudo-contractible, then it is not approximately biprojective,
since $\ell^{1}(\mathbb{N_{\vee}},w)$ has a unit $\delta_{1}$, see \cite[page 43]{dales}. Let $A$ be a Banach algebra with a bounded left approximate identity, with $\Delta(A)\neq \emptyset$.
Then use the similar arguments as in the previous Proposition  $A\otimes_{p}\ell^{1}(\mathbb{N_{\vee}},w)$ and $A\oplus \ell^{1}(\mathbb{N_{\vee}},w)$
are never approximately biprojective.
\end{Remark}
%------------------------------------------------------------------------------------------------------------------------------------------
\begin{small}

\end{small}

\end{document}